\documentclass[preprint,12pt]{elsarticle}

\usepackage[english]{babel}                                                                                                      
\usepackage[utf8]{inputenc}                                                                                                                                 \usepackage[all]{xypic}
\usepackage{amsmath,amsthm,amssymb,amsbsy,amsfonts,amstext,eufrak,latexsym, verbatim}

\newtheorem{thm}{Theorem}[section]                                                                                    
\newtheorem{lem}[thm]{Lemma}                                                                              
                                                                                               
\newtheorem{defn}{Definition}[section]                                                                                            
\newtheorem{cor}{Corollary}[section]                                                                                            
\theoremstyle{remark}                                                                                                            
                                                                                                
\theoremstyle{proposition}                                                                                                        
\newtheorem{proposition}{Proposition}[section]

\journal{Indagationes Mathematicae}

\begin{document}

\begin{frontmatter}



\title{Some Contributions on $P_F$-frames}

\author[1]{Mack Matlabyana\corref{cor1}%
}
\ead{marck.matlabyana@ul.ac.za}
\author[1]{Thabo Ngoako}
\ead{thabo.ngoako@ul.ac.za}
\author[1]{Hlengani Siweya}
\ead{hlengani.siweya@ul.ac.za}
\cortext[cor1]{Corresponding author}


\affiliation[1]{organization={Department of Mathematics and Applied Mathematics, University of Limpopo},
            addressline={Private Bag X1106}, 
            city={Sovenga},
            postcode={0727}, 
            state={Limpopo},
            country={South Africa}}

\begin{abstract}
The concept of $P_F$-frames was introduced by Ngoako \cite{lde} as a point-free extension of $P_F$-spaces. We observe that the open cozero quotient of a $P_F$-frame is itself a $P_F$-frame. The class of $P_F$-frames contains the class of $P$-frames and is, in turn, contained in the class of $F$-frames. We show that a frame $L$ is a $P_F$-frame if and only if $\beta L$ is a $P_F$-frame. Moreover, $P_F$-frames are precisely those essential $P$-frames that are also $F$-frames. Lastly, we provide a characterization of $P_F$-frames via $z$-, $d$-, and $z_r$-ideals.
\end{abstract}




\end{frontmatter}



\section{Introduction}
Throughout this paper, all spaces are assumed to be completely regular Hausdorff (Tychnorff), and the term ``ring '' means a commutative ring with identity. First, recall that whenever two zero sets whose union is all of a topological space $X$, at least one of them is open, we call it a $P_F$-space. $P_F$-spaces were introduced by Azarpanah \emph{et al} \cite{81} as one of the generalizations of $P$-spaces. These spaces lie between $P$-spaces and $F$-spaces. Following their definition in spaces we study $P_F$-frames which turns to be a spatial frame that means the conditions of the structure of $P_F$-spaces to $P_F$-frames is preserved. $P_F$-frames were first studied by Ngoako \cite{lde} in his Master's dissertation. As indicated in the abstract the class of $P_F$-frames contain the class of $P$-frames and is in turn contained in the class of $F$-frames. The class of $P_F$-frames is not contained in the class of basically disconnected frames and neither does it contain the class of basically disconnected frames. In this paper we extend the result of \cite{81} to point-free setting.

\vspace*{0.5cm}

The paper is organized as follows. In section 2, we provide foundational concepts (properties), definitions and notations needed for this study. In section 3, we study $P_F$-frames and some examples. In section 4, we show that $P_F$-frames are preserved by dense $coz$-onto frame homomorphism (Proposition 4.1) and reflected by nearly open $coz$-codense frame homomorphisms (Proposition 4.2). In section 5, we characterize $P_F$-frames in terms of $P$-ideals (Proposition 5.3), in terms of $z$-ideals (Proposition 5.8), in terms of $d$-ideals (Corollary 5.5) and in terms of $z_r$-ideals (Corollary 5.6).

\section{Preliminaries}
For a general theory of frames we refer to \cite{23,21,22}, and for more algebraic treatment of the subject our standard reference, notations and terms on $C(X)$ (denotes the ring of real valued continuous functions on a topological space $X$) (see \cite{13}). Good references about $\mathcal{R}L$ (denotes the ring of real-valued continuous functions on a frame $L$) are \cite{15,dfg}. We use the notation of Banaschewski \cite{dfg}. Here we collect a few facts that
will be relevant for our discussion. Recall that a \emph{frame} is a complete lattice $L$ in which the infinite distributive law
\[
a\wedge \bigvee S=\bigvee \{a\wedge s\mid s\in S\}
\] holds for all $a\in L$ and $S\subseteq L$. We denote the top element and the bottom element of $L$ by $1_L$ and $0_L$, respectively, dropping the decorations if $L$ is clear from the context. An element $a$ is \emph{rather below} an element $b$ (written $a\prec b$) if there is an element $t\in L$ such that $a\wedge t=0$ and $t\vee b=1$. We call such an element $t$ a \textit{separating} element. On the other hand, an element $a$ is \emph{completely below} an element $b$ (written $a\prec\!\prec b$) if there is a sequence $c_q$ indexed by rationals $\mathbb{Q}\cap [0,1]$ such that $c_0=a,\; c_1=b$ and $c_r\prec c_s$ whenever $r<s$. A frame $L$ is said to be \emph{regular} (resp. \emph{completely regular}) if for every $a\in L$, $a=\bigvee \{x\in L\mid x\prec a\}$ (resp. $a=\bigvee \{x\in L\mid x\prec\!\prec a\}$). Throughout our frames are assumed to be completely regular. A frame $L$ is called \textit{normal} if, given $a, b\in L$ with $a\vee b=1$, we can find $c,d\in L$ such that $c\wedge d=0$ and $a\vee c=1=b\vee d$. A map between frames which preserves finite meets, including the top element, and arbitrary joins, including the bottom
element is called a \emph{frame homomorphism}. It is \textit{dense} in case it maps only the bottom element to the bottom element. It is \textit{codense} if it maps only the top element to the top element. By a \textit{quotient} of a frame $L$, we mean an onto homomorphic image of $L$. That is $M$ is a quotient of $L$ precisely if there is an onto frame homomorphism $\pi:L\rightarrow M$. In such case $\pi$ is called a \textit{quotient map}. We call $D\subseteq L$ a \textit{downset} if $x\in D$ and $y\leq x$ implies $y\in D$. For any $a\in L$, we write
$\downarrow\!a=\{x\in L\mid x\leq a\}$, that is a downset. We note that $\downarrow\!a$ is a frame whose bottom element is $0\in L$ and top element $a$. This frame is in fact the quotient of $L$ via the map $L\rightarrow \downarrow\!a$, given by $x\mapsto a\wedge x$ and it is known as \textit{open quotient}. We call $U\subseteq L$ an \textit{upset} if $x\in U$ and $y\geq x$ implies that $y\in U$.  An ideal $J$ of a frame $L$ is said to be completely regular if for each $x\in J$ there exists $y\in J$ such that $x\prec\!\prec y$. For a completely regular $L$, the frame of its completely regular ideals is denoted by $\beta L$. The join map $\beta{L}\rightarrow L$ is dense onto and referred to as the Stone-$\check{C}$ech compactification of $L$. We write $r_L$ for the \textit{right adjoint} of the join map $\vartheta:\beta L\rightarrow L$ given by
\[
r_L(x)=\{a\in L\mid a\prec\!\prec x\}.
\]
Recall that $r_L$ preserves $\prec\!\prec$, and $\bigvee r_L(x)=x$ for any $x\in L$. Regarding the \textit{frame of reals} $\mathcal{L}(\mathbb{R})$ and the $f$-ring $\mathcal{R}L$, we have the \textit{cozero map} $coz:\mathcal{R}L\rightarrow L$, given by
$$coz\varphi=\bigvee\{\varphi(p,0)\vee \varphi(0,q)\mid p,q\in \mathbb{Q}\}.$$
The properties of the cozero map that we shall frequently use are:
\begin{enumerate}
\item $coz\gamma\delta=coz\gamma \wedge coz\delta$,
\item $coz(r+\delta)\leq coz\gamma\vee coz\delta$,
\item $\varphi\in\mathcal{R}L$ is invertible if and only if $coz\varphi=1$,
\item $coz\varphi=0$ if and only if $\varphi=0$, and
\item $coz(\gamma^{2}+\delta^{2})=coz\gamma\vee coz\delta$.
\end{enumerate}
A \textit{cozero element} of $L$ is an element of the form $coz\varphi$ for some $\varphi\in\mathcal{R}L$. It is shown in \cite{lm} that $a\in L$ is a cozero element if and only if there exists a sequence $(a_n)$ such that $a_n\prec\!\prec a$ for each $n$ and $a=\bigvee a_n$. The \textit{cozero part} of $L$, denoted by $CozL$, is the regular sub-$\sigma$-frame consisting of all the cozero elements of $L$, and the $CozL$ is normal. We refer to \cite{lm} for general properties of cozero elements and cozero parts of frames. We highlight the following:
 \begin{enumerate}
 \item If $a\prec\!\prec b$, there is a cozero element $s$ such that $a\prec\!\prec s\prec\!\prec b$.
 \item If $a\prec\!\prec b$, there is a cozero element $s$ such that $a\wedge s=0$ and $s\vee b=1$.
 \end{enumerate}The \emph{pseudocomplement} of an element $a$ is an element $$a^*=\bigvee \{x\in L\mid x\wedge a=0\}.$$ However $a\vee a^{*}=1$ does not hold in general. In the case where $a\vee a^{*}=1$, we say $a$ is \textit{complemented}. An element $a\in L$ is \textit{dense} if $a^*=0$, and \textit{regular} if $a=a^{**}$ for $a\in L$. The map $r_L$ preserves pseudocomplements, as a result of which the symbol $r_L(a)^{*}$ is not ambiquous. A frame homomorphism $\pi: L\rightarrow M$ is \textit{coz-onto} if for every $c\in CozM$ there exists $d\in CozL$ such that $\pi(d)=c$. It is \textit{coz-codense} if the only cozero element it maps to the top element is the top  element, and \textit{nearly open} if $\pi(c^{*})=\pi(c)^{*}$ for every $c\in L$. An element $\varphi\in\mathcal{R}L$ is \textit{bounded} if there exist $p,q\in\mathbb{Q}$ such that $\varphi(p,q)=1$. The subring of $\mathcal{R}L$ consisting of its bounded elements is denoted by $\mathcal{R}^{*}L$. $L$ is \textit{pseudocompact} if $\mathcal{R}L=\mathcal{R}^{*}L$. A quotient $\pi:L\rightarrow M$ of $L$ is a \textit{$C^{*}$-quotient} if for every bounded frame homomorphism $g:\mathcal{L}(\mathbb{R})\rightarrow M$, there is a frame homomorphism $\tilde{g}:\mathcal{L}(\mathbb{R})\rightarrow L$ such that $g=h\tilde{g}$, and this captures the notion of $C^{*}$-embedded subspaces.
\section{$P_F$-frames and related frames}

In this section, we study $P_F$-frames which are point-free extension of $P_F$-spaces and give some of frame-theoretic characterizations. Azarpanah \textit{et al} \cite{81} calls $X$  a $P_F$-space if for any two zero-sets in $X$ whose union is all of $X$, at least one of them is open. This  is equivalent to say for any two disjoint cozero-sets in $X$,  at least one of them is closed. This is extended to frames as follows.                                                                                                                                                                                                                                                            
\begin{defn}\textup{ A frame $L$ is said to be a $P_F$-\textit{frame} if whenever $a, b\in CozL$ such that $a\wedge b=0$, then at least one of them is complemented.}                                                                                                                                                                                                                                                                                    
\end{defn}                                                                                                                                                                                                                                                                                        

Next, we show that the open quotient of cozero elements of a $P_F$-frame is also a $P_F$-frame.
\begin{proposition} If $L$ is a $P_F$-frame, then $\downarrow\!\!a$ is a $P_F$-frame for each $a\in CozL$.
\end{proposition}
\begin{proof} Let $c, d\in Coz(\downarrow\!\!a)$ such that $c\wedge d=0_{\downarrow a}$. Then $c, d\in CozL$ such that $c\wedge d=0$. The frame $L$ is a $P_F$-frame, so at least one of them is complemented, say $c$, that is, $c^{*}\wedge c=0$ and $c^{*}\vee c=1$. Since $c^{\odot}=(c^{*}\wedge a)\in\downarrow\!\!a$ (see \cite{661}), we want to show that $c^{\odot}\in Coz(\downarrow\!\!a)$. Now,

\begin{center}{$( c\vee c^{\odot})=c\vee(c^{*}\wedge a)= (c\vee c^{*})\wedge(c\vee a)=1\wedge a= a=1_{\downarrow a}$,}\end{center}

and
\begin{center}{$c\wedge c^{\odot}=c\wedge (c^{*}\wedge a)=(c\wedge c^{*})\wedge a=0\wedge a=0_{\downarrow a}$.}\end{center}

Thus $c^{\odot}\in Coz(\downarrow\!\!a)$. Therefore, $c$ is complemented in $Coz(\downarrow\!\!a)$.
\end{proof}

\vspace*{0.1in}
We recall from \cite{15} that a frame $L$ is a $P$-\textit{frame} if every $a\in CozL$ is complemented in $L$. The following lemma shows that the class of $P_F$-frames contains the class of $P$-frames.
\begin{lem} Every $P$-\textit{frame} is a $P_F$-frame.\label{tt}                                                                                                                                                                                                                                  
\end{lem}                                                                                                                                                                                                                                                                                         \begin{proof} Let $a,b\in CozL$ such that $a\wedge b=0$. Since $L$ is a $P$-frame, hence we are done.
\end{proof}
We observe that the converse is not true. Suppose $L$ is a $P_F$-frame. We want to show that $L$                                                                                                                                                                                                                                            
is not a $P$-frame. Let $a,b\in CozL$ be such that $a\wedge b=0$. Since $L$ is $P_F$-frame. Thus, $a$ is complemented or $b$ is complemented or both are
complemented. Consider whether only $a$ is complemented, hence $b$ is not complemented. Similarly, if only $b$ is complemented, then $a$ is not complemented.
Thus, $L$ is not a $P$-frame.
\vspace*{0.1in}                                                                                                                                                                                                                                                                                   

We recall from \cite{15,62,661} that a frame $L$ is said to be
\begin{enumerate}
\item \textit{basically disconnected} if for any $a\in CozL$, $a^{*}\vee a^{**}=1$.
\item \textit{cozero complemented} if for each $a\in CozL$, there exist $b\in CozL$ such that $a \wedge b=0$ and $a\vee b$ is dense.
\item \textit{almost $P$-frame} if for every $a\in CozL$, $a=a^{**}$.
\end{enumerate}

Ball and Walters-Wayland \cite[Proposition 8.4.7]{15} showed that a frame is a $P$-frame if and only if it is basically disconnected almost $P$-frame. The following result is stronger than the result of Ball and Walters-Wayland. To see this, we observe from Matlabyana \cite{661} that basically disconnected frames are  cozero complemented. See also from Dube and Nsonde-Nsayi \cite{4} that every cozero complemented almost $P$-frame is a $P$-frame.
\begin{proposition} A frame is a $P$-frame if and only if it is a cozero complemented almost $P$-frame.\label{thabo}                                                                                                                                                                              
\end{proposition}                                                                                                                                                                                                                                                                                 

\begin{proof} The left to right implication is immediate from \cite[Proposition 8.4.7]{15}. Conversely, suppose that $a\in CozL$. We want to show that $a$ is complemented.
Since $L$ is cozero complemented, it follows that there exists $b\in CozL$ such that $a\wedge b=0$ and $a\vee b$ is dense. Then
 $$a^{*}\wedge b^{*}=(a\vee b)^{*}=0\Rightarrow  (a^{*}\wedge b^{*})^{*}=0^{*}=1.$$ Since $L$                                                                                                                                   
is an almost $P$-frame, it follows that $a=a^{**}$ and $a\vee b^{**}=1$. Thus, $a$ is complemented. Therefore, $L$ is a $P$-frame.                                                                                                                                                                      
\end{proof}

The authors in \cite{81} gave an example of a $P_F$-space which is not basically disconnected and an example of a basically disconnected space which is
 not a $P_F$-space. In the context of classical topology, $P_F$-spaces and basically disconnected spaces are incomparable.  This will also hold in the larger terrain of point-free topology. The following corollary follows immediately from \cite[Proposition 8.4.7]{15}, Proposition \ref{thabo} and Lemma \ref{tt}.

\begin{cor} Every basically disconnected (cozero complemented) almost $P$-frame is a $P_F$-frame.\label{jkj}                                                                                                                                                                                      
\end{cor}                                                                                                                                                                                                                                                                                         Next, we give a condition for a $P_F$-frame to be basically disconnected. The following proposition is, of course, an extension of \cite[Proposition 2.7 (1)]{81}.
\begin{proposition} Every cozero complemented $P_F$-\textit{frame} is basically disconnected. \label{a}                                                                                                                                                                                           
\end{proposition}                                                                                                                                                                                                                                                                                 

\begin{proof} Suppose that $L$ is a cozero complemented $P_F$-frame, we want to show that it is basically disconnected. Let
$a\in CozL$. The frame $L$ is cozero complemented, so there is $b\in CozL$ such that $a\wedge b=0$ and $a\vee b$ is dense. Again, $L$ is a $P_F$-frame, at least one of them is
complemented, say $a$. Since, again, $a$ is complemented hence, immediately, $a^{*}$ is the complement of $a$ and $a^{**}=a$.                                                                                                                                                                                               
\end{proof}

It is shown in  \cite[Proposition 2.2]{zfg} that a frame  $L$ is an \textit{$O_z$- frame}
 if and only if every regular element of $L$ is a cozero element.  Using this characterization,  Matlabyana \cite{661} has shown that an $O_z$-frame is cozero complemented. The following corollary is apparent from Matlabyana's result and Corollary \ref{jkj}.
\begin{cor} Every frame that is both an $O_z$-frame and an almost $P$-frame is a $P_F$-frame.\label{asa}                                                                                                                                                                                          
\end{cor}                                                                                                                                                                                                                                                                                         

Recall from  \cite{661} that a frame $L$ is said to satisfy \textit{countable chain condition ($ccc$)} if every collection of pairwise disjoint elements of $L$ is countable. It is shown in the same article that a frame  satisfying $ccc$ is cozero complemented. From,  Corollary \ref{jkj}, the following is immediate.

\begin{cor} Every frame that is an almost $P$-frame with $ccc$  is a $P_F$-frame.\label{aba}                                                                                                                                                                                                      
\end{cor}
We recall from \cite{15} that a frame $L$ is said to be an \textit{F-frame} if $L\rightarrow \downarrow\! a$ is a $C^{*}$-quotient for each $a\in CozL$. For the following result, we show that the class of $P_F$-frames is contained in the class of $F$-frames by using the fact that a frame $L$ is an $F$-frame if and only if whenever $a,b\in CozL$ such that $a\wedge b=0$ there exist $c,d\in CozL$ such that $c\vee d=1$ and $c\wedge a=0=d\wedge b$ (see \cite[Proposition 8.4.10]{15}).                                                                                                                                                                                 
\begin{proposition} Every $P_F$-\textit{frame} is an $F$-frame.\label{aa1}                                                                                                                                                                                                                        
\end{proposition}                                                                                                                                                                                                                                                                                 

\begin{proof} Let $a, b\in CozL$ such that $a\wedge b=0$. $L$ is a $P_F$-frame, so at least one is complemented, say $a$. Then $a\vee a^{*}=1$. Then $a^{*}$ as
a complemented element is a cozero element. By normality of $CozL$, there exist $c, d\in CozL$ such that $c\wedge d=0$ and $c\vee a=1=d\vee a^{*}$. Now                                                                                                                                           
$c\vee a=1\Rightarrow a^{*}\prec c$. Also $c\vee d\geq a^{*}\vee d=1$, and also $c\wedge d=0$ and $c\vee a=1$ $\Rightarrow$ $d\prec a$. Then                                                                                                                                                          
$d\wedge b\leq a\wedge b=0$.  Also $c\wedge d=0$ and $d\vee a^{*}=1 \Rightarrow c\prec a^{*}$, then $c\wedge a\leq a^{*}\wedge a=0$. Therefore                                                                                                                                                    
$c\wedge a=0=d\wedge b$. Thus $L$ is an $F$-frame.                                                                                                                                                                                                                                                
\end{proof}                                                                                                                                                                                                                                                                                       We recall from \cite{15} that a frame $L$ is said to be an $F'$-\textit{frame} if $a\wedge b=0$ for $a,b\in CozL$ implies $a^{*}\vee b^{*}=1$. It is well-known from the same article that every $F$-frame is an $F'$-frame. We recall from \cite[Proposition 8.4.10]{15} that a frame $L$ is said to be a \textit{quasi} $F$-\textit{frame} if and only if whenever $a, b\in CozL$ such that $a\wedge b=0$ and $a\vee b$ is dense, then there exist $c, d\in CozL$ such that $c\vee d=1$ and $c\wedge a=0=d\wedge b$. It is well-known from the same article that every $F$-frame is a quasi $F$-frame. Therefore, the following are immediate from Proposition \ref{aa1}.                                                                                                                                                                                         
\begin{cor} Every $P_F$-frame is an $F'$-frame.                                                                                                                                                                                                                                                   
\end{cor}                                                                                                                                                                                                                                                                                         
\begin{cor} Every $P_F$-frame is a quasi $F$-frame.                                                                                                                                                                                                                                               
\end{cor}

\section{Transportations  of $P_F$-frames}
In this section, we show that $P_F$-frames are preserved under the transportation of $coz$-onto dense frame homomorphisms. On the other hand, $P_F$-frames are also reflected
by nearly open frame homomorphisms which are $coz$-codense.                                                                                                                                                                                                                                       

\begin{proposition}  Let $\pi:L\rightarrow M$ be a $coz$-$onto$, dense frame homomorphism. If $L$ is a $P_F$-frame, then so is $M$.                                                                                                                                                                 
\end{proposition}                                                                                                                                                                                                                                                                                 

\begin{proof} Let $a, b\in CozM$ such that $a\wedge b=0$. Since $\pi$ is $coz$-$onto$, there exist $x, y\in CozL$ such that $\pi(x)=a$ and $\pi(y)=b$. Now,                                                                                                                                              
$$\pi(x\wedge y)=\pi(x)\wedge \pi(y)=a\wedge b=0_M.$$ By density of $\pi$, $x\wedge y=0_L$. But $L$ is a $P_F$-frame, so at least one of them is
complemented, say $x$. Frame homomorphisms are well known for preserving complemented elements. Thus, $M$ is a $P_F$-frame.
\end{proof}                                                                                                                                                                                                                                                                                       

\begin{proposition} Let $\pi:L\rightarrow M$ be a nearly open and $coz$-codense frame homomorphism. If $M$ is a $P_F$-frame, then so is $L$.                                                                                                                                                        
\end{proposition}                                                                                                                                                                                                                                                                                 

\begin{proof} Let $a, b\in CozL$ such that $a\wedge b=0$. Then since frame homomorphisms preserve cozero elements, it follows that $\pi(a)$ and $\pi(b)$ are cozero                                                                                                                                  
elements in $M$. Furthermore,                                                                                                                                                                                                                                                                     
\begin{center} {$\pi(a)\wedge \pi(b)=\pi(a\wedge b)=\pi(0_L)=0_M$.}                                                                                                                                                                                                                                       
\end{center}                                                                                                                                                                                                                                                                                      
$M$  is a $P_F$-frame, so at least one of $\pi(a)$ or $\pi(b)$ is complemented, say $\pi(a)$. That is, \begin{center}{$\pi(a)\wedge \pi(a)^{*}=0_M$                                                                                                                                      
\textup{ and }$\pi(a)\vee \pi(a)^{*}=1_M$.}
\end{center} Now $\pi(a\wedge a^{*})=\pi(a)\wedge \pi(a^{*})\leq \pi(a)\wedge \pi(a)^{*}=0_M$. By nearly openness of $\pi$, we have $\pi(a\vee a^{*})=\pi(a)\vee                                                                                                                                             
\pi(a^{*})=\pi(a)\vee \pi(a)^{*}=1_M$. By $coz$-codense of $\pi$, $a\vee a^{*}=1_L$. Hence $a$ is a complemented element in $L$.  Thus, $L$ is a $P_F$-frame.                                                                                                                                                      
\end{proof}                                                                                                                                                                                                                                                                                       
$P_F$-frames relate to the Stone-\textup{$\check{C}$}ech compactification as follows.

\begin{proposition} A frame $L$ is a $P_F$-frame if and only if $\beta{L}$ is a $P_F$-frame.                                                                                                                                                                                                               
\end{proposition}                                                                                                                                                                                                                                                                                         
\begin{proof}                                                                                                                                                                                                                                                                                     
Assume that $L$ is a $P_F$-frame. Take $I, J\in Coz\beta L$ such that $$I\wedge J=0_{\beta L}.$$ Put                                                                                                                                                                                                
$a\equiv \bigvee I\in CozL$ and                                                                                                                                                                                                                                                    
$b\equiv\bigvee J\in CozL$.                                                                                                                                                                                                                                                                                     
Now, the map $\vartheta: \beta L\rightarrow L$ is a frame homomorphism, so                                                                                                                                                                                                                                 
\begin{center}{$a\wedge b=\bigvee I\wedge\bigvee J=\bigvee (I\wedge J)=\vartheta(I\wedge J)=0_L$.}                                                                                                                                                                                                        
\end{center}                                                                                                                                                                                                                                                                                      
Now, $L$ is a $P_F$-frame, so one of $a$ or $b$ is complemented, say $a$. That is, $a\vee a^{*}=1_L$. That is                                                                                                                                                                                      
\begin{center}{$1_{\beta L}=r_L(a\vee a^{*})=r_L(a)\vee r_L(a^{*})\leq r_L(a)\vee r_L(a)^{*}=I\vee I^{*}$.}                                                                                                                                                                                     
\end{center}                                                                                                                                                                                                                                                                                      
Thus $I$ is complemented in $\beta L$ and hence $\beta L$ is a $P_F$-frame.                                                                                                                                                                                                                       
\vspace*{0.1in}                                                                                                                                                                                                                                                                                   

Conversely, suppose that $\beta L$ is a $P_F$-frame. Take $a, b\in CozL$ such that $a\wedge b=0_L$. Put $a\equiv\bigvee I\in CozL$ and $b\equiv\bigvee J\in CozL$. Then by density of the join map it follows that $I\wedge J=0_{\beta L}$.                                                                                                                                                                                                                                                                                                                                                                                                                                                                                                                     
Now, $\beta L$ is a $P_F$-frame, so there exist $I^{*}\in Coz\beta L$ such that $I\vee I^{*}=1_{\beta L}$. Thus                                                                                                                                                                                    
\begin{center}{$\vartheta(I\vee I^{*})=\bigvee I\vee \bigvee I^{*}=(\bigvee I)\vee (\bigvee I)^{*}=a\vee a^{*}=1_L$.}                                                                                                                                                                                     
\end{center}                                                                                                                                                                                                                                                                                      
Hence $a$ is complemented and thus $L$ is a $P_F$-frame.                                                                                                                                                                                                                                          
\end{proof}                                                                                                                                                                                                                                                                                       

\section{Ring-theoretic characterizations of $P_F$-frames.}
In this section, we characterize $P_F$-frames in terms of the ring $\mathcal{R}L$. According to Osba \textit{et al} \cite{100}, a ring $R$ is \textit{von Neumann local} (henceforth abbreviated $VN$-local) if for each $a\in R$, at least one of $a$ or $1-a$ has a $VN$-inverse, that is, there is $b\in R$ such that at least one of $a$ or $1-a$, we have $aba=a$ or $(1-a)b(1-a)=(1-a)$. The following proposition is motivated by \cite[Lemma 4.5]{12}.                                                                                                                                                                                                                         
\begin{proposition} If $L$ is a $P_F$-frame, then $\mathcal{R}{L}$ is a $VN$-local ring.\label{qp}                                                                                                                                                                                                         
\end{proposition}                                                                                                                                                                                                                                                                                 
\begin{proof} Let $a, b\in CozL$ and $a\wedge b=0$. Then $a\wedge b$ is complemented. There exists $c\in CozL$ such that $(a\wedge b)\vee c=(a\vee c)\wedge(b\vee                                                                                                                                 
c)=1$, then $a\vee c=1$ and $b\vee c=1$. The frame $L$ is a $P_F$-frame, say $a$ is complemented. Therefore, $\mathcal{R}{L}$ is a $VN$-local ring.                                                                                                                                                                     
\end{proof}

Essential $P$-frames were introduced by Dube \cite{12} and it captures the notion of essential $P$-spaces in a point-free setting. Recall that a space $X$ is an essential $P$-space if it has at most one point which fails to be a $P$-point (see \cite{7,100}). In this article, we call a frame $L$ to be an \textit{essential $P$-frame} if at most one $a\in CozL$ fails to be  complemented. We recall from \cite{lkm} that a proper ideal $I$ in a ring $R$ is said to be a \textit{semiprime} ideal if, whenever $J^{n}\subset I$ for an ideal $J$  of $R$ and some positive integer $n$, then $J\subset I$. Equivalently, $I$ is semiprime if $a^{2}\in I$  implies $a\in I$. This definition can be given in terms of $\mathcal{R}L$ where $R$ can be replaced with $\mathcal{R}L$. We also recall from \cite{12} that an ideal $I$ of $\mathcal{R}L$ is called a \textit{$z$-ideal} if for any $\alpha\in \mathcal{R}L$ and $\varphi\in I$, $coz\alpha=coz\varphi$ implies that $\alpha\in I$ (see \cite{kak}). The following proposition is a point-free version of \cite[Proposition 2.2]{81}. The next result is not about $P_F$-frames. However, it is related to $P_F$-frames as will be shown later.

\begin{proposition} The following statements are equivalent.\label{mmm}                                                                                                                                                                                                                           
\begin{enumerate}                                                                                                                                                                                                                                                                                 
\item [\textup{(1)}] A frame $L$ is an essential $P$-frame.                                                                                                                                                                                                                                                          
\item [\textup{(2)}] Of any two comaximal ideals of $\mathcal{R}(L)$, one is a $z$-ideal.                                                                                                                                                                                                                        
\item [\textup{(3)}] Of any two comaximal principal ideals of $\mathcal{R}(L)$, one is semiprime.                                                                                                                                                                                                                
\item [\textup{(4)}] For any $a, b\in CozL$ such that $a\vee b$ is complemented , then one of them is complemented.                                                                                                                                                                                              
\end{enumerate}                                                                                                                                                                                                                                                                                   
\end{proposition}

\begin{proof}(1) $\Rightarrow$ (2): Suppose that $I$ and $J$ are ideals of $\mathcal{R}L$ such that $I+J=\mathcal{R}L$. We need to show that either $I$ is a                                                                                                                                  
$z$-ideal or $J$ is a $z$-ideal. Let $\tau\in \mathcal{R}L$ and $\varphi\in I$ such that $coz\tau=coz\varphi$. $L$ is an essential $P$-frame so there is a                                                                                                                                        
$c\in CozL$ such that $c\notin I$, but $r(c)\subseteq I$. Hence $c\ne coz\varphi=coz\tau$ does not belong to $coz\varphi$ and $coz\tau$. Since $c$ is the only cozero element which is not complemented, it follows that both $coz\varphi$ and $coz\tau$ are cozero complemented. Therefore, either $\varphi\leq\tau$ or $\tau\leq\varphi$. Then either $\varphi\in J$ or $\tau\in I$ and hence either $I$ or $J$ is a $z$-ideal.                                                                                                                                                     
\vspace*{0.1in}                                                                                                                                                                                                                                                                                   

(2) $\Rightarrow$ (3): Let $\langle \phi\rangle$ and $\langle\varphi \rangle$ be two comaximal principal ideals of $\mathcal{R}L$, one is a $z$-ideal. Say, $\langle \phi\rangle$  is a $z$-ideal and let $\alpha^{2}\in \langle \phi\rangle$. Then for any $\omega\in \mathcal{R}L$ such that $coz(\alpha^{2})=coz\omega$. By
$z$-ideal, we have $coz\alpha=coz\alpha\wedge coz\alpha=coz(\alpha^{2}).$ Hence $\alpha\in \langle \phi\rangle$.                                                                                                                                                                          
\vspace*{0.1in}                                                                                                                                                                                                                                                                                   

(3) $\Rightarrow$ (1): Let $I$ and $J$ be two comaximal ideals of $\mathcal{R}L$ such that $I=\langle \varphi\rangle$ and $J=\langle\delta\rangle$. Then there exist $\tau=\varphi^{n}+\delta^{m}$, for some $n, m\in \mathbb{N}$. We assume without loss of generality that $\tau\ne \varphi^{n}$ and $\tau\ne\delta^{m}$. That is, $\tau\notin I$ and $\tau\notin J$. Then $c=coz\tau\notin I$ and $c=coz\tau\notin J$. Now, $$c=coz\tau=coz(\varphi^{n}+\delta^{m})= coz|\varphi^{n}|\vee coz|\delta^{m}|.$$ So, we can write  $\tau=\delta(\varphi^{p}+\delta^{q})$; where $\delta=\varphi^{k}$ and $\delta=\delta^{s}$, so that $\delta\in I\cap J$ and  for some $p, q, k, s\in \mathbb{N}$. Hence $$r(c)=r(coz\tau)=r(coz(\delta(\varphi^{p}+\delta^{q})))=r(coz(\delta(\varphi^{p})))\vee r(coz(\delta(\delta^{q}))).$$ By hypothesis (one is semiprime), say $I$ is semiprime. Hence $c\in CozL$, $c\notin I$ but $r(c)\subseteq I$, showing that $L$ is an essential $P$-frame.                                                                                                                                                                                                                                                                   
\vspace*{0.1in}                                                                                                                                                                                                                                                                                  

(1) $\Leftrightarrow$ (4): Let $a, b\in CozL$ such that $a\vee b=1$. Then it is immediate by the definition of an essential $P$-frame that at least one is                                                                                                                                    
complemented because $L$ has at most one cozero element which is not complemented. Conversely, suppose that $a\vee b=1$ such that $a$ is complemented (by hypothesis), then $a\vee a^{*}=1$. Now, $a\vee b=1$ implies  that $a^{*}\leq b$. If $b$ is not complemented, then $b\vee b^{*}\ne1$. Suppose on contrary that $u$ is another cozero element such that $u\vee a=1$. Again by hypothesis, $a$ is complemented. Then $a\vee a^{*}=1$, and $u\vee a=1$ implies that $a^{*}\leq u$. Therefore, $u=b$. Thus $L$ has only one cozero element which is  not complemented. Hence $L$ is an essential $P$-frame.                                                                                                                                                                                                       
\end{proof}

\textbf{Remark :} (2) $\Leftrightarrow$ (3): Also hold from the fact that every $z$-ideal is semiprime. Hence for any two comaximal principal ideals, one being a $z$-ideal. Thus, the one that is a $z$-ideal is semiprime.                                                                                                                                                                                                  
\vspace*{0.1in}

We give a brief account of $P$-ideals and also provide the characterization of $P_F$-frames associated with $P$-ideals. $P$-ideals were firstly introduced and studied in $C(X)$ (topological space $X$ is said to be completely regular Hausdorff space) by Rudd \cite{SFG}. We recall from the same article that a nonzero ideal $I$ of $C(X)$ is called a $P$-ideal if every proper prime ideal of $I$ is maximal in $I$. We have the definition of $P$-ideal in terms of $\mathcal{R}L$ as follows, a nonzero ideal $I$ in $\mathcal{R}L$ is said to be a \textit{$P$-ideal} if every proper prime ideal in $I$ is maximal in $I$. We show that an ideal $I$ in $\mathcal{R}L$ is  a $P$-ideal if and only if  for each $f\in I$, then $cozf$ is
complemented (see \cite[Theorem 1.5]{SFG}). The following proposition follows from the fact that $C(X)$ is a von Neumann regular ring if and only if all of its pure ideals are $P$-ideals (see \cite[Theorem 2.2 and Theorem 2.3]{ltd} and \cite{gcg}).                                                                                                                                                                                                                                       
\begin{proposition} An ideal $I$ is a $P$-ideal in $\mathcal{R}L$ if and only if for every $f \in I$, $cozf$ is complemented in $L$.
\end{proposition}
\begin{proof} We want to prove that an ideal $I$ of a distributive lattice $L$ is a $P$-ideal if and only if for every $f\in I$, $cozf$ is complemented. If $I$ is a $P$-ideal and $f\in I$, then there exists $g\in I$ such that $fg=0$ and $g=1$ on supp $f$. That is, $f\vee g=1$. Hence $0=coz(fg)=cozf\wedge cozg$ and $1=coz(f\vee g)=cozf\vee cozg$. This shows that $cozf$ is complemented.
\vspace*{0.1in}

Conversely, if for every $f \in I$, $coz(f)$ is complemented in $L$, and $f,g\in L$  such that $f \wedge g \in I$. We need to show that $f$ or $g$ belongs to $I$. Let $h = (cozf \vee cozg)^{*}$. Since $cozf$ and $cozg$ are complements in $L$, we have: $$cozf \wedge h = 0 \text{ and } cozg \wedge h = 0.$$
Thus, $cozf \leq h^{*}$ and $cozg \leq h^{*}$. It follows that $h^{*} \in I$ since $I$ is an ideal. We have:
$$h \wedge (f \wedge g) = (cozf \vee cozg)^{*} \wedge (f \wedge g) = ((cozf)^{*} \wedge (f \wedge g)) \wedge ((cozg)^{*} \wedge (f \wedge g)) = 0.$$
This shows that $h^{*}$ is an upper bound for $cozf$ and $cozg$. Since $cozf$ and $cozg$ are complements in $L$, it follows that $h$ is a lower bound for $f$ and $g$. Thus, we have:
$$f \vee g \leq h^{*} \in I$$
Therefore, $I$ is a $P$-ideal.
\end{proof}

From Proposition \ref{mmm}, we can say that we are now ready to give the characterization about $P_F$-frames, and the following proposition is an extension of \cite[Theorem 2.4]{81}.                                                                                                                                                                               

\begin{proposition} For a frame $L$, the following statements are equivalent.\label{jik}                                                                                                                                                                                                                     
\begin{enumerate}                                                                                                                                                                                                                                                                                 
\item [\textup{(1)}] $L$ is a $P_F$-frame.                                                                                                                                                                                                                                                                       
\item [\textup{(2)}] If $a, b\in CozL$ such that $a\wedge b$ is complemented, then at least one of them is complemented.                                                                                                                                                                                         
\item [\textup{(3)}] Of any two ideals of $\mathcal{R}{L}$ whose product is a $P$-ideal, at least one is a $P$-ideal.                                                                                                                                                                                           
\item [\textup{(4)}] Of any two principal ideals of $\mathcal{R}{L}$ whose product (intersection) is zero, at least one is semiprime.                                                                                                                                                                            
\item [\textup{(5)}] Of any two principal ideals of $\mathcal{R}{L}$ whose product (intersection) is semiprime, at least one is semiprime.                                                                                                                                                                       
\item [\textup{(6)}] $L$ is an essential $P$-frame which is also an $F$-frame.                                                                                                                                                                                                                                 
\end{enumerate}                                                                                                                                                                                                                                                                                   
\end{proposition}                                                                                                                                                                                                                                                                                 

\begin{proof}                                                                                                                                                                                                                                                                                     

(1) $\Rightarrow$ (2): Let $a, b \in CozL$ and $a\wedge b$ be complemented. If $a\wedge b=0$, then we are done. Now assume without loss of generality that                                                                                                                                        
$a\wedge b\ne0$. There exists $c\in CozL$ such that  $(a\wedge b)\vee c=(a\vee c)\wedge( b\vee c)=1$ implies $a\vee c=1$ and $b\vee c=1$. Now, if                                                                                                                                                  
$(a\wedge b)\wedge c=0$ then  $c\leq(a\wedge b)^{*}$. Put $c=a^{*}\vee b^{*}$. Now                                                                                                                                                                                                        \begin{equation*}
\begin{split}b^{*}\vee b & =b^{*}\vee[(a^{*}\wedge b)\vee(a\wedge b)]\\
& =[b^{*}\vee(a^{*}\wedge b)]\vee(a\wedge b)\\
& =[(b^{*}\vee a^{*})\wedge(b^{*}\vee b)]\vee(a\wedge b)\\
& =(b^{*}\vee a^{*})\vee(a\wedge b)\\
& =c\vee (a\wedge b)\\
& =(c\vee a)\wedge (c\vee b)\\
& =1.
\end{split}
\end{equation*}                                                                                                                                                                                                                                                                                      
 The fourth step holds because $a\wedge b\ne 0$. Hence $b$ is complemented as required.                                                                                                                                                                                                         
\vspace*{0.1in}                                                                                                                                                                                                                                                                                   

(2) $\Rightarrow$ (3): Let $I$ and $J$ be two ideals of $\mathcal{R}{L}$ whose product (intersection) is a $P$-ideal. Suppose $J$ is not a $P$-ideal, then there                                                                                                                                  
is $j\in J$ such that $coz(j)$ is not complemented. Since $ij\in IJ$ for each $i\in I$, $coz(ij)=cozi\wedge cozj$ is complemented ($IJ$ is a $P$-ideal). Now                                                                                                                                  
using (2), $cozi$  must be complemented for each $i\in I$, so $I$ is a $P$-ideal and we are done.                                                                                                                                                                                               
\vspace*{0.1in}                                                                                                                                                                                                                                                                                   

For intersection of two ideals $I$ and $J$ in $\mathcal{R}{L}$, where  $I\bigcap J$ is a $P$-ideal. Then we have $I\bigcap J=IJ$.                                                                                                                                                             
\vspace*{0.1in}                                                                                                                                                                                                                                                                                   

(3) $\Rightarrow$ (4): If the product or intersection of two principal ideals  $\langle f\rangle$ and $\langle g\rangle$ is zero, then $\langle f\rangle\langle                                                                                                                                   
g\rangle=\langle f\rangle\bigcap\langle g\rangle=\langle fg\rangle=\langle 0\rangle$ is a $P$-ideal. Now using (3), one of the principal ideals                                                                                                                                                   
$\langle f\rangle$  and $\langle g\rangle$ is a $P$-ideal whence it must be semiprime.

\vspace*{0.1in}                                                                                                                                                                                                                                                                                   
(4) $\Rightarrow$ (5): Let the product (intersection) of principal ideals $\langle f\rangle$ and $\langle g\rangle$ be semiprime. Then                                                                                                                                                          
$\langle f\rangle\langle g\rangle=\langle f\rangle\bigcap\langle g\rangle=\langle fg\rangle$ and $coz(fg)=cozf\wedge cozg$ is complemented. Take complemented                                                                                                                                 
element $cozt={coz(fg)}^{*}$ for some $t\in \mathcal{R}{L}$. Now $\langle tf\rangle\langle tg\rangle=\langle 0\rangle$ implies that either $\langle tf\rangle$                                                                                                                                  
or $\langle tg\rangle$ is semiprime (by (4)), say  $\langle tf\rangle$ is semiprime. This implies that $coz(tf)=cozt\wedge cozf$  is complemented. On the                                                                                                                                     
other hand $cozt\vee cozf=1$ implies that $cozf$ is complemented and hence the principal ideal $\langle f\rangle$ will be semiprime.                                                                                                                                                        

\vspace*{0.1in}                                                                                                                                                                                                                                                                                   
(5) $\Rightarrow$ (1): Let $cozf\wedge cozg=0$. Then $\langle f\rangle\bigcap\langle g\rangle=\langle 0\rangle$ is semiprime. Hence by (5), at least one of the
ideals $\langle f\rangle$ and $\langle g\rangle$ is semiprime. Hence either $cozf$ or $cozg$ is  complemented. Therefore $L$ is a $P_F$-frame.                                                                                                                                                
\vspace*{0.1in}                                                                                                                                                                                                                                                                                   

(1) $\Rightarrow$ (6): We only need to show that $L$ is an essential $P$-frame. Suppose on contrary that $a, b\in CozL$ which are not complemented with $a\ne b$. If
$a\wedge b=0$, then we are done. We assume without loss of generality that $a\wedge b\ne0$. If $a\wedge b$ is complemented, then at least one of them is complemented, a contradiction. If  $a\wedge b\in CozL$ is not complemented, then $(a\wedge b)\vee (a\wedge b)^{*}\ne1$. The frame $L$  is completely regular and hence $CozL$ generates $L$. There exists a cozero element $u\leq (a\wedge b)^{*}$. We
claim that $u=b^{*}\vee a^{*}$ is a non complemented cozero element, then $u^{*}=(b^{*}\vee a^{*})^{*}=b^{**}\wedge a^{**}$. Now                                                                                                                                                                  \begin{equation*}
\begin{split} u\vee u^{*} & =(b^{*}\vee a^{*})\vee (b^{**}\wedge a^{**})\\
& =(b^{*}\vee a^{*}\vee b^{**})\wedge(b^{*}\vee a^{*}\vee a^{**})\\
& =(b^{*}\vee b^{**})\wedge(a^{*}\vee a^{**})\\
& \ne 1.
\end{split}
\end{equation*}                                                                                                                                                                                                                                                                                      The third step holding because $a^{*}\leq b^{*}\vee b^{**}$ and $b^{*}\leq a^{*}\vee a^{**}$. Thus $u$ and $a\wedge b$ are disjoint cozero elements of $L$ which are not complemented. Hence a contradiction since $L$ is a $P_F$-frame. Hence $L$ has at most one cozero element which is not complemented and we are done.                                                                                                                                                                                                                                    
\vspace*{0.1in}                                                                                                                                                                                                                                                                                   

(6) $\Rightarrow$ (1): Suppose $a, b\in CozL$ such that $a\wedge b=0$. Since $L$ is an $F$-frame, there exist $c, d\in CozL$ such that $c\vee d=1$ and $c\wedge                                                                                                                                          
a=0=d\wedge b$. Now $c\wedge a=0$ implies $c\leq a^{*}$ and $d\wedge b=0$ implies $d=b^{*}$. Now $a\prec d$ and $b\prec c$ implies $a^{*}\vee d=1=b^{*}\vee c$.                                                                                                                                   
Then $a^{*}\vee b^{*}=1$. This shows that $a\wedge b$ is complemented. The frame $L$ is an essential $P$-frame, so at least one of them is complemented.

\end{proof}                                                                                                                                                                                                                                                                                       The following corollary follows immediately from \cite[Corollary 4.7]{12} and Proposition \ref{jik}.
\begin{cor} A normal frame $L$ is a $P_F$-frame if and only if $\mathcal{R}L$ is a $VN$-local ring.                                                                                                                                                                                                      
\end{cor}

Although the following proposition and corollary are weaker than \cite[Proposition 5.1 and Corollary 5.2]{12}, they are worth noting. We recall from \cite{1cl} that a frame $L$ is said to be \textit{strongly zero-dimensional} if and only if $a\prec\!\prec b$ implies that there exists $c\in L$ such that $a\leq c\leq b$ for some complemented $c$.                                                                                                                                                                    

\begin{proposition} A $P_F$-frame is strongly zero-dimensional.                                                                                                                                                                                                                                           
\end{proposition}

\begin{proof} Let $a\prec\!\prec b$ in a $P_F$-frame $L$. Take $c\in CozL$ such that $a\prec\!\prec c\prec\!\prec b$. If $c$ is complemented, then we are done. Suppose $c$ is not complemented. Take $s\in CozL$ such that $a\wedge s=0$ and $s\vee c=1$. Since $L$  is a $P_F$-frame  and $c$ is not complemented, thus $s$ is complemented. Hence $s^{*}$ is also complemented. Now $a\wedge s=0$ implies that $a\leq s^{*}$, and $s\vee c=1$ implies that $s^{*}\leq c$. Consequently, $a\leq s^{*}\leq b$. Thus $L$ is strongly zero-dimensional.
\end{proof}

\begin{cor} A $P_F$-space is strongly zero-dimensional.                                                                                                                                                                                                                                           
\end{cor}                                                                                                                                                                                                                                                                                         

The following proposition is an extension of \cite[Proposition 2.5]{81}.
\begin{proposition} Every compact $P_F$-frame is finite. More generally, every pseudocompact $P_F$-frame is finite.                                                                                                                                                                                
\end{proposition}                                                                                                                                                                                                                                                                                 

\begin{proof} Let $L$ be a compact $P_F$-frame and suppose on contrary, that $L$ is infinite. Then $L$ is a one-point compactification of an indiscrete frame                                                                                                                                     
which implies that $L$ is an $F$-frame. On the other hand $L$ contains a copy of $\beta\mathbb{N}$, because by hypothesis                                                                                                                                   
is an infinite $F$-frame which is impossible. So $L$ is finite and we are done. Whenever $L$ is pseudocompact then $\beta{L}$ must be finite as a consequence of the first part of the proof.
\end{proof}

Recall from \cite{2,17} that a space $X$ is an \textit{almost P-space} if every dense $G_\delta$-set of $X$ has a non-empty interior. For the construction of the definition below, we first recall the definition of an almost $P$-frame. For the equivalence part, we recall  from \cite{661} that a point $I$ of $\beta{L}$ is an \textit{almost $P$-point} if for any $\alpha\in \mathbf{M}^{I}=\{\varphi\in \mathcal{R}L\mid r_L(coz\varphi)\subseteq I\}$, $coz\alpha$ is not dense. This motivates us to note the following definition in terms of ideals.                                                                                                                                                                                                     
\begin{defn} \textup{A frame $L$ is \textit{essential almost $P$-frame} if there is at most one cozero element which is not regular. Equivalently, if there is at most one $I\in \beta{L}$ such that $\alpha\in \mathbf{M}^{I}$, $coz\alpha$ is dense.}
\end{defn}                                                                                                                                                                                                                                                                                        It is clear that every essential $P$-frame is an essential almost $P$-frame. Next, we show that for a basically disconnected frame to be a $P_F$-frame it needs to be essential almost $P$-frame, and the following proposition is an extension of \cite[Proposition 2.7 (2)]{81}.
\begin{proposition} Every basically disconnected essential almost $P$-frame is a $P_F$-frame.\label{ad}                                                                                                                                                                                           
\end{proposition}                                                                                                                                                                                                                                                                                 
\begin{proof} Let $a, b\in CozL$ be such $a\wedge b=0$. The frame $L$ is basically disconnected, so \begin{center}{$a^{*}\vee a^{**}=1=b^{*}\vee b^{**}$.}
\end{center} Furthermore, $L$ is almost
$P$-frame, so at least of $a, b\in CozL$ is regular. Thus, at least one of $a$ and or $b$ is complemented. Hence $L$ is a $P_F$-frame.                                                                                                                                                             
\end{proof}                                                                                                                                                                                                                                                                                       
 The following corollary is a point-free version of \cite[Corollary 2.8]{81}, and it follows from Proposition \ref{a}, Proposition \ref{ad} and the fact that every basically disconnected frame is cozero complemented (see \cite{661}).                                                                                                                                                                                                                                                                           
\begin{cor} A frame is a  cozero complemented $P_F$-frame if and only if it is a basically disconnected essential almost $P$-frame.\label{php}                                                                                                                                                     
\end{cor}                                                                                                                                                                                                                                                                                         

The following corollary which is an extension of \cite[Corollary 2.9]{81}, follows immediately from the fact that a frame $L$ is an $F$-frame if and only if each ideal of $\mathcal{R}L$ is a convex and, Proposition \ref{mmm}.
\begin{cor} A frame $L$ is a $P_F$-frame if and only if for any two comaximal principal ideals of $\mathcal{R}L$, one is semiprime and the other is convex.                                                                                                                                        
\end{cor}                                                                                                                                                                                                                                                                                         
In this paper, we prefer to use the word $d$-ideal instead of $z^{\circ}$-ideal. We recall from \cite{kak,1322} that an ideal $I$ in $\mathcal{R}L$ is said to be a $d$-ideal if and only if for any $\alpha\in \mathcal{R}L$ and $f\in I$, we have $coz\alpha\leq(cozf)^{**}$ implies $\alpha\in I$. The following lemmas are point-free extension of \cite[Lemma 2.1 and Lemma 3.1, respectively]{81}.
\begin{lem} Let $L$ be an essential $P$-frame and $I$ be an ideal of $\mathcal{R}L$. Whenever $\bigvee_{f\in I}cozf$ is complemented for each $f\in I$, then $I$ is a $d$-ideal and hence a $z$-ideal. In particular every free ideal of $\mathcal{R}L$ is a $d$-ideal. \label{pin}
\end{lem}

\begin{proof} In case $L$ is a $P$-frame, then every ideal of $\mathcal{R}L$ is a $z$-ideal by \cite[Proposition 3.9]{12}. Since every cozero element in $L$ is complemented, every ideal of $\mathcal{R}L$, where $L$ is a $P$-frame is a $d$-ideal. So, in case $L$ is $P$-frame, the proof is evident. Now assume a unique cozero element that fails to be complemented, say $c$. Let $\bigvee coz[I]$ be complemented for each $f\in I$. Then $c\ne \bigvee coz[I]$ and hence there exists $f\in I$ such that $c \neq cozf$. This means $cozf$ is complemented. Now suppose $cozg\leq (cozp)^{**}$  for $p\in I$ and $g\in\mathcal{R}L$. Clearly $cozg\leq (coz(p^{2}+f^{2}))^{**}$ and $coz(p^{2}+f^{2})$ is complemented and we have $cozg\leq coz(p^{2}+f^{2})$ (see \cite[Lemma 7.3.2]{kak}), $g$ is a multiple of $p^{2}+f^{2}$, that is, $g\in I$. Thus $I$ is a $d$-ideal.
\end{proof}

\begin{lem} If $L$ is a $P_F$-frame, then every ideal of $\mathcal{R}L$ which is not a $z$-ideal is contained in $\mathbf{M}_{a}=\{\varphi\in \mathcal{R}L\mid coz\varphi\leq a\}$ where $a\in CozL$ such that $a$ is not complemented. \label{pk}
\end{lem}

\begin{proof} By Lemma \ref{pin}, if a prime ideal of $\mathcal{R}L$ is not a $z$-ideal, then it must be fixed. For each $b\in L$ with $a\neq b$, we have $\mathbf{M}_b=\mathbf{O}_b$ and hence the only prime ideal in $\mathbf{M}_b$ is $\mathbf{M}_b$ itself which is a $z$-ideal. Hence the only candidates for prime ideals, which are not $z$-ideals, are those which are contained in $\mathbf{M}_a$.
\end{proof}

The frame $L$ is a $P$-frame if and only if every ideal of $\mathcal{R}L$ is a $z$-ideal (see \cite[Proposition 3.9]{12}). So in case $L$ is a $P$-frame, then $\mathcal{R}L$ has the property that for every two ideals whose sum (intersection) is a $z$-ideal, then both ideals are $z$-ideals. But if $L$ is even an essential $P$-frame, then $\mathcal{R}L$ may not have aforementioned property (see \cite{81}). The following proposition shows that if $L$ is an essential $P$-frame, then the sum (intersection) of two ideals  of $\mathcal{R}L$ may be a proper $z$-ideal but neither of which may be a $z$-ideal. The following proposition and corollary are point-free versions of Theorem 3.2 and Corollary 3.3 in \cite{81}.

\begin{proposition} The following statements are equivalent. \label{plq}
\begin{enumerate}
\item [\textup{(1)}] Of any two ideals in $\mathcal{R}L$ whose sum is a proper $z$-ideal, at least one is a $z$-ideal.
\item [\textup{(2)}] $L$ is a $P_F$-frame.
\item [\textup{(3)}] Of any two ideals in $\mathcal{R}L$ whose intersection is a $z$-ideal, at least one is a $z$-ideal.
\end{enumerate}
\end{proposition}

\begin{proof} (1) $\Rightarrow$ (2): First we show that $L$ is an $F$-frame. Suppose $K$ for some $K\in \beta{L}$ is not prime. Then using \cite[Corollary 3.8]{1} there are two prime ideals ideals $P$ and $Q$ in $\mathbf{M}^K$  which are not in a chain. Hence there exist $f\in Q\backslash P$ and $g\in P\backslash Q$. If we take $I=\langle g\rangle+ P$ and $J=\langle f\rangle+ Q$, then $I+J=\langle f\rangle+Q+\langle g\rangle+P$. Now $P+Q$ is a $z$-ideal by \cite[Theorem 3.2]{pnm}, so $I+J$ is a $z$-ideal. But neither $I$ nor $J$ is a $z$-ideal because $coz(g^{\frac{1}{3}})=cozg$, $g\in I$ but $g^{\frac{1}{3}}\notin I$. In fact if $g^{\frac{1}{3}}\in I=\langle g\rangle +P$, then $g^{\frac{1}{3}}=gt+h$ for some $t\in \mathcal{R}L$ and $h\in P$. Therefore $g^{\frac{1}{3}}(1-tg^{\frac{2}{3}})\in P\subseteq \mathbf{M}^{K}$ implies that $(1-tg^{\frac{2}{3}})\in P\subseteq\mathbf{M}^{K}$ which is impossible  because $tg^{\frac{2}{3}}\in \mathbf{M}^{K}$. Therefore $L$ must be an $F$-frame.

\vspace*{0.1in}

To prove that $L$ is an essential $P$-frame, suppose on contrary that $a, b\in CozL$ which are not complemented. Let $z\in CozL$ be different  from $a$ and $b$  and take non-negative functions $f\in \mathbf{M}_{a}\backslash\mathbf{O}_{a}$ and $g\in \mathbf{M}_{b}\backslash\mathbf{O}_{b}$ such that $cozf\vee cozg=1$ and $z\nleq cozf\wedge cozg$. Consider the ideals $I=\langle f\rangle\mathbf{O}_{z}$ and $J=\langle g\rangle\mathbf{O}_{z}$. Clearly $I+J=\mathbf{O}_{z}$ (note that if $t\in \mathbf{O}_{z}$, then $t=\frac{f}{f+g}t+\frac{g}{f+g}t$) and it is a proper $z$-ideal. We show that $I$  and $J$ are not $z$-ideals which contradicts our hypothesis. To see this since $z\nleq cozf\wedge cozg$, there exists $h\in \mathcal{R}L$ such that $z\geq cozh$ and $cozh\vee(cozf\wedge cozg)=1$. Now $fh\in \langle f\rangle\mathbf{O}_{z}$, $coz(fh)=coz(f^{\frac{1}{3}}h)$, but $f^{\frac{1}{3}}h\notin \langle f\rangle\mathbf{O}_{z}$. In fact if $f^{\frac{1}{3}}h\in \langle f\rangle\mathbf{O}_{z}$, then \\ $f^{\frac{1}{3}}h=fq$ for some $q\in \mathbf{O}_{z}$ and hence $f^{\frac{1}{3}}(h-f^{\frac{2}{3}}q)=0$. The latter equality implies that $coz(f^{\frac{1}{3}})\wedge coz(h-f^{\frac{2}{3}}q)=0$ and $coz(f^{\frac{1}{3}})\vee coz(h-f^{\frac{2}{3}}q)=1$, because $cozh\vee cozf=1$. This implies that $cozf$ is complemented which contradicts $f\in \mathbf{M}_{a}\backslash\mathbf{O}_{a}$. Similarly we may also show that $J$ is not a $z$-ideal and we get the desired contradiction. Thus $L$ is an essential $P$-frame.

\vspace*{0.1in}

(2) $\Rightarrow$ (1): Let $I$ and $J$ be two ideals in $\mathcal{R}L$ and $I+J$ be a proper $z$-ideal. Suppose on contrary that $I$ and $J$ are not $z$-ideals. Then using  \cite[Corollary 7.2.2]{kak}  there are prime ideals $P$ and $Q$ with P minimal over $I$ and Q minimal over $J$  which are not $z$-ideals. So $P$ and $Q$ are contained in $\mathbf{M}_{a}$ by Lemma \ref{pk}, where $a\in CozL$ such that $a$ is not complemented. Since $L$ is an $F$-frame, $P$ and $Q$  are in a chain by \cite[Corollary 3.8]{1}, say $P\subseteq Q$. Accordingly $I+J\subseteq P+Q \subseteq Q$. On the other hand if $T$ is a prime ideal such that $I+J\subseteq T\subseteq Q$, then $J\subseteq T\subseteq Q$ implies $T=Q$, since $Q$ is minimal over $J$. This implies that $Q$ is a minimal prime ideal over $I+J$ which must be a $z$-ideal by \cite[Corollary 7.2.2]{kak}, a contradiction.

\vspace*{0.1in}

(1) $\Leftrightarrow$ (3): Since every semiprime principal ideal of $\mathcal{R}L$ is a $z$-ideal, using part (5) of Proposition \ref{jik}, (3) $\Rightarrow$ (2) and so it implies (1). For the converse, suppose that $I$ and $J$ are two ideals of $\mathcal{R}L$ such that $I\cap J$ is a $z$-ideal. Suppose on contrary that $I$ and $J$ are not $z$-ideals. Hence there are prime ideals $P$ and $Q$ minimum over $I$ and $J$, respectively, which are not $z$-ideals by \cite[Corollary 7.2.2]{kak}. Using Lemma \ref{pk}, $P$ and $Q$ are contained in $\mathbf{M}_{a}$, where $a\in CozL$ such that $a$ is not complemented. Therefore $P$ and $Q$ are in a chain by \cite[Corollary 3.8]{1}, say $P\subseteq Q$. Since $I\cap J+\mathbf{O}_{a}\subseteq P+\mathbf{O}_{a}\subseteq \mathbf{M}_{a}$, the sum $I\cap J+\mathbf{O}_{a}$ of two $z$-ideals is a proper $z$-ideal.  Accordingly to \cite[Lemma 4.3]{1} $I\cap J+\mathbf{O}_{a}$ is prime because it contains the prime ideal $\mathbf{O}_{a}$. On the other hand $I\cap J+\mathbf{O}_{a}=(I+\mathbf{O}_{a})\cap (J+\mathbf{O}_{a})$ is prime, so $I+\mathbf{O}_{a}$ and $J+\mathbf{O}_{a}$ are in  a chain, say $I+\mathbf{O}_{a}\subseteq J+\mathbf{O}_{a}$. Accordingly, $I\cap J+\mathbf{O}_{a}=I+\mathbf{O}_{a}$ that is a prime $z$-ideal. But $I\subseteq I+\mathbf{O}_{a}\subseteq P$ implies that $I+\mathbf{O}_{a}=P$ since $P$ is minimal over $I$ and this means that $P$ is a $z$-ideal, a contradiction. In case $J+\mathbf{O}_{a}\subseteq I+\mathbf{O}_{a}$, we get again a contradiction by a similar proof.
\end{proof}

Every $d$-ideal is a $z$-ideal, and in von Neumann regular rings $d$-ideals coincide with
$z$-ideals (see \cite{1322}).  The characterizations in Proposition \ref{plq} hold with “$z$-ideal” replaced by
“$d$-ideal”.

\begin{cor} The following statements are equivalent.\label{geo}
\begin{enumerate}
\item [\textup{(1)}] The sum of every two $d$-ideal in $\mathcal{R}L$ is a $d$-ideal or all of $\mathcal{R}L$ and of any two ideals in $\mathcal{R}L$ whose sum is a proper $d$-ideal, at least one is a $d$-ideal.
\item [\textup{(2)}] $L$ is a $P_F$-frame.
\item [\textup{(3)}] Of any two ideals in $\mathcal{R}L$ whose intersection is a $d$-ideal, at least one is a $d$-ideal.
\end{enumerate}
\end{cor}

\begin{proof} (1) $\Leftrightarrow$ (2): As in the first part of the proof of Proposition \ref{plq}, if $\mathbf{O}^{K}$ for some $K\in \beta{L}$ is not prime, then there are two prime ideals $P$ and $Q$ in $\mathbf{M}^{K}$ which are not in a chain. As in Proposition \ref{plq}, take $f\in P\backslash Q$, $g\in Q\backslash P$, $I=\langle g\rangle+P$ and $J=\langle f\rangle+Q$, then $$I+J=\langle f\rangle+Q+\langle g\rangle+ Q= P+Q.$$ Now using \cite[Corollary 3.13]{jaz} or \cite[Theorem 5.1]{jeg}, $P+Q$ is a prime $d$-ideal, because $L$ is a quasi $F$-frame (by our hypothesis, the sum of every two $d$-ideals in $\mathcal{R}L$ is a $d$-ideal or all of $\mathcal{R}L$). But as we observed in the proof of Proposition \ref{plq}, $I$ and $J$ are not even $z$-ideals. This implies that $L$ is an $F$-frame. The same proof of Proposition \ref{plq}  also shows that $L$ is an essential $P$-frame. Conversely, suppose that $I$ and $J$ are two ideals in $\mathcal{R}L$ whose sum is a $d$-ideal, but neither of which is a $d$-ideal.  As in the proof of (2) $\Rightarrow$ (1) of Proposition \ref{plq} and using \cite[Corollary 7.2.2]{kak}, there are prime ideals in $\mathbf{M}_{a}$ which are in a chain and not $d$-ideals. Using the same proof and the same corollary we get a contradiction.

\vspace*{0.1in}

(1) $\Leftrightarrow$ (3): Since every semiprime principal ideal of $\mathcal{R}L$ is a $d$-ideal, using part (5) of Proposition \ref{jik}, (3) $\Rightarrow$ (2) and so it implies (1). Conversely, suppose, on contrary that $I\cap J$ is a $d$-ideal, but neither $I$ nor $J$ is a $d$-ideal. Hence  there are prime ideals $P$ and $Q$ minimal over $I$ and $J$ respectively which are not $d$-ideals by \cite[Corollary 7.2.2]{kak}. Using Lemma \ref{pin} $P$ and $Q$ are contained in $\mathbf{M}_{a}$, where $a\in CozL$ is such that $a$ is not complemented. Similar proof of  Proposition \ref{plq}, $P$ and $Q$  are in chain, say $P\subseteq Q$. Since $L$ is a quasi $F$-frame, $I\cap J+\mathbf{O}_{a}$ is a proper $d$-ideal. On the other hand $I\cap J+\mathbf{O}_{a}$ is prime, since it contains the prime ideal $\mathbf{O}_{a}$. The rest of the proof is similar to that of (1) $\Rightarrow$ (3) of Proposition \ref{plq} and we show that $I+\mathbf{O}_{a}=P$ and $J+\mathbf{O}_{a}=Q$ are $d$-ideals which is a contradiction.
\end{proof}

We recall from \cite{twalathesis} that an ideal $I$ of $\mathcal{R}L$ is an \textit{$r$-ideal} if for each $\alpha\in \mathcal{R}L$ and $\varphi\in r(L)$ ($r(L)$ denotes a set of all regular elements of $\mathcal{R}L$), $\alpha\varphi\in I$ implies $\alpha\in I$. The authors called an ideal to be \textit{$z_r$-ideal} if it is both a $z$-ideal and $r$-ideal. They specified that $d$-ideals are $z_r$-ideals. Lastly, $z_r$-ideals lie between $d$-ideals and $z$-ideals. Hence, the following corollary follows immediately from Corollary \ref{geo} and Proposition \ref{plq}.

\begin{cor} The following statements are equivalent.
\begin{enumerate}
\item [\textup{(1)}] The sum of every two $z_r$-ideal in $\mathcal{R}L$ is a $z_r$-ideal or all of $\mathcal{R}L$ and of any two ideals in $\mathcal{R}L$ whose sum is a proper $z_r$-ideal, at least one is a $z_r$-ideal.
\item [\textup{(2)}] $L$ is a $P_F$-frame.
\item [\textup{(3)}] Of any two ideals in $\mathcal{R}L$ whose intersection is a $z_r$-ideal, at least one is a $z_r$-ideal.
\end{enumerate}
\end{cor}

\textbf{Acknowledgments}: The second named author acknowledges the Department of Mathematics and Applied Mathematics for allowing him to pursue Master's studies at the University of Limpopo under the supervision of the first named and third named authors. This is the last chapter of the second author's Masters dissertation.\\

\textbf{Declaration of generative AI and AI-assisted technologies in the writing process:} During the preparation of this work the authors used chatGPT in order to fix grammar. After using this tool/service, the authors reviewed and edited the content as needed and takes full responsibility for the content of the publication.    




\end{document}